\documentclass[a4paper,10pt,oneside,reqno]{amsart}
\usepackage{amssymb,amsthm,amsmath}
\usepackage{ifthen}
\usepackage{graphicx}
\nonstopmode
 \numberwithin{equation}{section}
\newtheorem{thm}{Theorem}[section]
\newtheorem{cor}[thm]{Corollary}
\newtheorem{lem}[thm]{Lemma}
\newtheorem{prop}[thm]{Proposition}
\newtheorem{defn}[thm]{Definition}

\theoremstyle{definition}
\newtheorem{rmk}[thm]{Remark}
\newtheorem{example}[thm]{Example}

{\qed\bigskip}

\newcounter{alphabet}

\ifx\undefined\bysame
\newcommand{\bysame}{\leavevmode\hbox to3em{\hrulefill}\,}
\fi

\begin{document}
\baselineskip=21pt
\markboth{} {}

\title[Semi-continuous $G$-frames in Hilbert spaces]
{Semi-continuous $G$-frames in Hilbert spaces}

\author{Anirudha Poria}

\address{Department of Mathematics, School of Engineering and Applied Sciences, Bennett University, Greater Noida, Uttar Pradesh, India}

\address{Department of Mathematics, Indian Institute of Science, Bengaluru, Karnataka, India}

\email{anirudhap@iisc.ac.in}

\begin{abstract} 
In this paper, we introduce the concept of semi-continuous $g$-frames in Hilbert spaces. We first construct an example of semi-continuous $g$-frames using the Fourier transform of the Heisenberg group and study the structure of such frames. Then, as an application we provide some fundamental identities and inequalities for semi-continuous $g$-frames. Finally, we present a classical perturbation result and prove that semi-continuous $g$-frames are stable under small perturbations.

\end{abstract}
\date{\today}
\maketitle
\def\BC{{\mathbb C}} \def\BQ{{\mathbb Q}}
\def\BR{{\mathbb R}} \def\BI{{\mathbb I}}
\def\BZ{{\mathbb Z}} \def\BD{{\mathbb D}}
\def\X{{\mathcal X}} \def\BB{{\mathbb B}}
\def\BS{{\mathbb S}} \def\BH{{\mathcal H}}
\def\BE{{\mathbb E}}  \def\BK{{\mathcal K}}
\def\BN{{\mathbb N}}  \def\H{{\mathbb H}}
\def\g{{\mathcal{G}_j}}
\def\gxj{{\mathcal{G}_{x,j}}}
\def\gtr{{\mathcal{G}_j^*}}
\def\gt{{\mathcal{\tilde{G}}_j}}
\def\ga{{\varGamma_j}}
\def\gm{{\Gamma_{x,j}}}
\vspace{-.5cm}

\section{Introduction}
Discrete and continuous frames arise in many applications in mathematics and, in particular, they play important roles in scientific computations and digital signal processing. The concept of a frame in Hilbert spaces has been introduced in 1952 by Duffin and Schaeffer \cite{duf52}, in the context of nonharmonic Fourier series (see \cite{you01}). After the work of Daubechies et al. \cite{dau86} frame theory got considerable attention outside signal processing and began to be more broadly studied (see \cite{chr13, gro01}). A frame for a Hilbert space is a redundant set of vectors in Hilbert space which provides non-unique representations of vectors in terms of frame elements. The redundancy and flexibility offered by frames has spurred their application in several areas of mathematics, physics, and engineering such as wavelet theory, sampling theory, signal processing, image processing, coding theory and many other well known fields. Applications of frames, especially in the last decade, motivated researchers to find some generalization of frames like continuous frames \cite{ali93, kai94}, $g$-frames \cite{sun06}, Hilbert$-$Schmidt frames \cite{por17, pori17}, $K$-frames \cite{gav12, gav15} and etc. Our main purpose in this paper is to study a generalization of frames, namely semi-continuous $g$-frames, which are natural generalizations of $g$-frames and continuous $g$-frames. 
We investigate the structure of semi-continuous $g$-frames and establish some identities and inequalities of these frames. Also, we present a perturbation result and discuss the stability of the perturbation of a semi-continuous $g$-frame.

Throughout this paper, $\BH$ and $\BK$ are two Hilbert spaces; $J$ is a countable index set; $(\X, \mu)$ is a measure space with positive measure $\mu$; $\{\BK_x\}_{x \in \X}$ is a family of closed subspaces of $\BK$; $\mathcal{L}(\BH, \BK_x)$ is the collection of all bounded linear operators from $\BH$ into $\BK_x$; if $\BK_x=\BH$ for any $x \in \X$, we denote $\mathcal{L}(\BH, \BK_x)$ by $\mathcal{L}(\BH)$.

We recall that a family $\{f_j\}_{j \in J}$ in $\BH$ is called a (discrete) {\it frame} for $\BH$, if there exist constants $0 < A \leq B < \infty$ such that
\begin{equation*}
A \Vert f \Vert^2 \leq \sum_{j \in J } \vert \langle f,f_j \rangle \vert^2 \leq  B \Vert f \Vert^2, \quad \forall f \in \BH.
\end{equation*}

The concept of the discrete frame was generalized to continuous frame by Kaiser \cite{kai94} and independently by Ali et al. \cite{ali93}. A family of vectors $\{\psi_x\}_{x \in \X} \subseteq \BH$ is called a {\it continuous frame} for $\BH$ with respect to $(\X, \mu)$, if $\{\psi_x\}_{x \in \X}$ is weakly-measurable, i.e., for any $f \in \BH, \; x \to \langle f, \psi_x \rangle$ is a measurable function on $\X$, and if there exist two constants $A,B > 0$ such that
\[ A \Vert f \Vert^2 \leq \int_{\X} \vert \langle f,\psi_x \rangle \vert^2 d \mu(x) \leq  B \Vert f \Vert^2, \quad \forall f \in \BH. \]
Continuous frames have been widely applied in continuous wavelets transform \cite{ali00} and the short-time Fourier transform \cite{gro01}. We refer to \cite{ask01, for05, gab03} for more details on continuous frames.

The notion of a discrete frame was extended to $g$-frame by Sun \cite{sun06}, which generalized all the existing frames such as bounded quasi-projectors \cite{for04}, frames of subspaces \cite{cas04}, pseudo-frames \cite{li04}, oblique frames \cite{chr04}, etc. $G$-frames are natural generalizations of frames as members of a Hilbert space to bounded linear operators. Let $\{ \BK_j: j \in J \} \subset \BK$ be a sequence of Hilbert spaces. A family $\{ \Lambda_j \in \mathcal{L}(\BH,\BK_j):j \in J \}$ is called a {\it $g$-frame}, for $\BH$ with respect to $\{ \BK_j: j \in J \}$ if there are two constants $A, B>0$ such that
\begin{equation*}
A \Vert f \Vert^2 \leq \sum_{j \in J }  \Vert \Lambda_j(f)  \Vert^2 \leq  B \Vert f \Vert^2, \quad \forall f \in \BH.
\end{equation*} 

The continuous $g$-frames were proposed by Dehghan and Hasankhani Fard in \cite{deh08}, which are an extension of $g$-frames and continuous frames. A family $\{ \Lambda_x \in \mathcal{L}(\BH,\BK_x):x \in \X \}$ is called a {\it continuous $g$-frame} for $\BH$ with respect to $(\X, \mu)$, if $\{\Lambda_x: x \in \X \}$ is weakly-measurable, i.e., for any $f \in \BH, \; x \to \Lambda_x(f)$ is a measurable function on $\X$, and if there exist two constants $A,B > 0$ such that
\[ A \Vert f \Vert^2 \leq \int_{\X} \Vert \Lambda_x(f) \Vert^2 d \mu(x) \leq  B \Vert f \Vert^2, \quad \forall f \in \BH. \]
Notice that if $\X$ is a countable set and $\mu$ is a counting measure, then the continuous $g$-frame is just the $g$-frame. By the Riesz representation theorem, for any $\Lambda \in \mathcal{L}(\BH, \BC)$, there exist a $h \in \BH$, such that $\Lambda(f)=\langle f,h \rangle$ for all $f \in \BH$. Hence, if $\BK_x=\BC$ for any $x \in \X$, then the continuous $g$-frame  is equivalent to the continuous frame.

This paper is organized as follows.  After the introduction, in Section \ref{sec2}, we introduce the semi-continuous $g$-frames in Hilbert spaces and construct an example using the Fourier transform of the Heisenberg group. Then we study the structure of semi-continuous $g$-frames using shift-invariant spaces. In Section \ref{sec3}, we first list some fundamental identities and inequalities of discrete frames just for the contrast to the main results of this section. Then we derive some important identities and inequalities of semi-continuous $g$-frames. Finally, in Section \ref{sec4}, we present a classical perturbation result and prove that semi-continuous $g$-frames are stable under small perturbations.

\section{Semi-continuous g-frames}\label{sec2}
Let $\{\BK_{x,j}:x \in \X, j \in J \} \subset \BK$ be a family of Hilbert spaces. 
\begin{defn}
A family $\{ \Lambda_{x,j} \in \mathcal{L}(\BH,\BK_{x,j}):x \in \X, j \in J \}$ is called a  semi-continuous $g$-frame for $\BH$ with respect to $(\X, \mu)$, if $\{\Lambda_{x,j}: x \in \X, j \in J\}$ is weakly-measurable, i.e., for any $f \in \BH$ and any $j \in J$, the function $x \to \Lambda_{x,j}(f)$ is measurable on $\X$, and if there exist two constants $A,B > 0$ such that
\begin{equation}\label{eq01}
A \Vert f \Vert^2 \leq \int_{\X} \sum_{j \in J} \Vert \Lambda_{x,j}(f) \Vert^2 d \mu(x) \leq  B \Vert f \Vert^2, \quad \forall f \in \BH.
\end{equation}
\end{defn}
If only the right-hand inequality of \eqref{eq01} is satisfied, we call $\{\Lambda_{x,j}: x \in \X, j \in J\}$ the semi-continuous $g$-Bessel sequence for $\BH$ with respect to $(\X, \mu)$ with Bessel bound $B$.
\begin{rmk}
If $0<\mu(\X)< \infty$, and for any fixed $x \in \X$, the family $\{\Lambda_{x,j}:j \in J\}$ is a $g$-frame for $\BH$ with respect to $\{\BK_{x,j}:j \in J\}$, then $\{\Lambda_{x,j}: x \in \X, j \in J\}$ is a semi-continuous $g$-frame for $\BH$ with respect to $(\X, \mu)$. Moreover if $|J|< \infty$, and for any fixed $j \in J$, the family $\{\Lambda_{x,j}:x \in \X \}$ is a continuous $g$-frame for $\BH$, then $\{\Lambda_{x,j}: x \in \X, j \in J\}$ is a semi-continuous $g$-frame for $\BH$ with respect to $(\X, \mu)$.
\end{rmk}
In the following, we shall construct an example of such frames using the Fourier transform of the Heisenberg group.

\subsection{Heisenberg Group}
The Heisenberg group $\H$ is a Lie group whose underlying manifold is $\BR^3$. We denote points in $\H$ by $(p, q, t)$ with $p, q, t \in \BR$, and define the group operation by
\begin{equation}
(p_1, q_1, t_1) (p_2, q_2, t_2)=(p_1 + p_2, q_1 + q_2, t_1 + t_2 + \frac{1}{2} (p_1 q_2 - q_1 p_2)).
\end{equation}
It is easy to verify that this is a group operation, with the origin $0 = (0, 0, 0)$ as the
identity element. Notice that the inverse of $(p, q, t)$ is given by $(-p,-q,-t)$. The Haar measure on the group $\H=\BR^3$ is the usual Lebesgue measure.

The irreducible representations of the Heisenberg group has been identified by all non-zero elements in $\BR^*(=\BR \setminus \{0 \})$ (see \cite{fol95}). Indeed, for any $\lambda \in \BR^*$, the associated irreducible representation $\rho_{\lambda}$ of $\H$ is equivalent to {\it Schr{\"o}dinger} representation into the class of unitary operators on $L^2(\BR)$, such that for any $(p, q, t) \in \H$ and $f \in L^2(\BR)$, the operator $\rho_{\lambda}(p, q, t)$ is defined by
\begin{equation}
\rho_{\lambda}(p, q, t)f(x)=e^{i \lambda t}  e^{i \lambda (px+\frac{1}{2}(pq))} f(x+q).
\end{equation}
 It is easy to see that $\rho_{\lambda}(p, q, t)$ is a unitary operator satisfying the homomorphism property:
 \[ \rho_{\lambda}((p_1, q_1, t_1) (p_2, q_2, t_2))= \rho_{\lambda}(p_1, q_1, t_1) \rho_{\lambda}(p_2, q_2, t_2). \]
Thus each $\rho_{\lambda}$ is a strongly continuous unitary representation of $\H$. By Stone and von Neumann theorem (\cite{fol95}), $\{ \rho_{\lambda}: \lambda \in \BR^* \}$ are all the infinite dimensional irreducible unitary representations of $\H$, whose set has non-zero Plancherel measure. The measure $|\lambda| d\lambda$ is the Plancherel measure on the dual space $\widehat{\H}\; (\cong \BR^*)$ of $\H$, and $d \lambda$ is the Lebesgue measure on $\BR^*$. For $\varphi \in L^2(\H)$ and $\lambda \in \BR^*$, we denote $\widehat{\varphi}(\lambda)$ the operator-valued Fourier transform of $\varphi$ at a given irreducible representation $\rho_\lambda$, which is defined by
\begin{equation}
\widehat{\varphi}(\lambda)=\int_{\H} \varphi(x) \rho_\lambda(x) dx.
\end{equation}
The operator $\widehat{\varphi}(\lambda)$ is a unitary map on $L^2(\BR)$ into $L^2(\BR)$, such that for any $f \in L^2(\BR)$
\[ \widehat{\varphi}(\lambda)f(y)=\int_{\H} \varphi(x) \rho_\lambda(x) f(y) dx. \]
Therefore $\widehat{\varphi}(\lambda)$ belongs to $L^2(\BR) \otimes L^2(\BR)$. If $\varphi \in  L^2(\H)$, $\widehat{\varphi}(\lambda)$ is actually a Hilbert$-$Schmidt operator on $L^2(\BR)$ and from the Plancherel theorem we have
\begin{equation}
\int_{\H} |\varphi(x)|^2 dx= \int_{\BR^*} \Vert \widehat{\varphi}(\lambda) \Vert^2_{H.S.} \; |\lambda| d\lambda,
\end{equation}
the norm $\Vert \cdot \Vert_{H.S.}$ denotes the Hilbert$-$Schmidt norm in $L^2(\BR) \otimes L^2(\BR)$. The proof of the Plancherel theorem for the Heisenberg group can be found in \cite{gel77}, and for more general groups, see \cite{fol95}.

To construct our example of semi-continuous $g$-frames, we shall define another unitary operator as follows.

Let $\Pi:=[0,1]$ and $\mathfrak{L}:=\ell^2(\BZ, L^2(\BR)\otimes L^2(\BR))$ be the Hilbert space of all sequences with values in the space $L^2(\BR)\otimes L^2(\BR)$, i.e.,
\[ \mathfrak{L}= \Big\{ \{a_n\}_{n \in \BZ}: \; a_n \in L^2(\BR)\otimes L^2(\BR) \mathrm{\;and\;} \sum_{n \in \BZ} \| a_n \|^2_{H.S.}< \infty \Big\}. \] 
\begin{lem}\label{lem1}
For any $\sigma \in \Pi$, let $T_\sigma: L^2(\H) \rightarrow \mathfrak{L}$ given by $T_\sigma f(j)=|\sigma+j|^{\frac{1}{2}} \widehat{f}(\sigma+j)$. Then $T_\sigma$ is well-defined and $\sum_{j \in \BZ}|\sigma+j| \; \|\widehat{f}(\sigma+j)\|^2_{H.S.}<\infty.$
\end{lem}
\begin{proof}
Let $f \in  L^2(\H).$ Using Plancherel theorem and an application of periodization method, we obtain
\begin{eqnarray*}
\|f\|^2_{L^2(\H)}=\int_{\BR^*} \|\widehat{f}(\lambda)\|_{H.S.}^2 |\lambda| d\lambda 
&=& \int_{\sigma \in \Pi} \sum_{j \in \BZ} |\sigma+j| \; \|\widehat{f}(\sigma+j)\|^2_{H.S.} d\sigma \\
&=& \int_{\sigma \in \Pi} \sum_{j \in \BZ} \| T_\sigma f(j) \|_{H.S.}^2 d\sigma.
\end{eqnarray*} 
Hence, the result follows from the fact that $f \in  L^2(\H)$.
\end{proof}
\begin{example}
Consider $\X=\Pi$ and $J=\BZ$. For any $\sigma \in \Pi$ and $j \in \BZ$, define $\Lambda_{\sigma, j}: L^2(\H) \rightarrow L^2(\BR) \otimes L^2(\BR)$ as $\Lambda_{\sigma, j}(f)=T_\sigma f(j).$ Then for every $f \in L^2(\H)$, using Lemma \ref{lem1} we get 
\begin{eqnarray*}
\int_{\sigma \in \Pi} \sum_{j \in \BZ} \| \Lambda_{\sigma, j}(f) \|_{H.S.}^2 d\sigma &=&\int_{\sigma \in \Pi} \sum_{j \in \BZ} \| T_\sigma f(j) \|_{H.S.}^2 d\sigma \\
&=& \int_{\sigma \in \Pi} \sum_{j \in \BZ} \left\| |\sigma+j|^{\frac{1}{2}} \widehat{f}(\sigma+j) \right\|_{H.S.}^2 d\sigma \\
&=& \| f \|^2_{L^2(\H)}.
\end{eqnarray*}
Therefore $\{ \Lambda_{\sigma, j}: \sigma \in \Pi, j \in \BZ \}$ is a semi-continuous $g$-frame with frame bounds $A=B=1$.
\end{example}
\begin{cor}
Let $0<\mu(\X)<\infty$. For any fixed $\sigma \in \X$, let $\{ \Lambda_{\sigma, j}: j \in J \}$ be a $g$-frame for $L^2(\H)$. Then $\{ \Lambda_{\sigma, j}: \sigma \in \X, j \in J \}$ is a semi-continuous $g$-frame for $L^2(\H)$ with respect to $(\X,\mu)$ with unified frame bounds multiplied by $\mu(\X)$.
\end{cor}
\begin{proof}
Since $\{ \Lambda_{\sigma, j}: j \in J \}$ be a $g$-frame for $L^2(\H)$, there exist constants $A,B>0$ such that  
\[ A\| f \|^2_{L^2(\H)} \leq \sum_{j \in J} \|\Lambda_{\sigma, j}(f) \|^2 \leq B \| f \|^2_{L^2(\H)}, \quad  \forall f \in L^2(\H).\]
Taking integral from all sides of the preceding inequality, we obtain
\[ A\mu(\X) \| f \|^2_{L^2(\H)} \leq \int_{\X} \sum_{j \in J} \|\Lambda_{\sigma, j}(f) \|^2 d\mu(\sigma) \leq B \mu(\X) \| f \|^2_{L^2(\H)}, \quad  \forall f \in L^2(\H).\]
Hence, the result follows.
\end{proof}
Now, we shall define shift-invariant spaces and give an example.
\begin{defn}
Let $\Gamma$ be a countable subset of $\H$. A subspace $\mathcal{V} \subset L^2(\H)$ is called $\Gamma$-invariant if $L_{\gamma}\phi \in \mathcal{V}$ for all $\gamma \in \Gamma$ and all $\phi \in \mathcal{V}$, where $L_{\gamma}\phi(w)=\phi(\gamma^{-1}w), \; w \in \H.$ If $\Gamma$ is a discrete subset of $\H$, then $\mathcal{V}$ is called shift-invariant.
\end{defn}
\begin{example}
Let $\phi \in L^2(\H)$ and $\Gamma$ be a lattice. Then the space $\langle \phi \rangle_{\Gamma}$ generated by $\Gamma$-shifts of $\phi$ is a shift-invariant space. 
\end{example}
Before we prove the main result of this section, we first need the following.

Let $T: L^2(\H) \rightarrow L^2\left(\Pi, \mathfrak{L} \right)$. Then for any $\sigma \in \Pi$ and $j \in \BZ$, $Tf(\sigma)(j) \in L^2(\BR)\otimes L^2(\BR)$. By Lemma \ref{lem1}, it is clear that $Tf(\sigma)=T_{\sigma}f$. Let 
\[ \Gamma=\Gamma_1\Gamma_0=\big\{xz \in \H: x \in \Gamma_1, z \in \Gamma_0 \big\},\] 
where $\Gamma_1$ be any discrete subset of $\H$ and $\Gamma_0$ be the lattice of integral points in $\BZ$. Then for $y \in \H$ and $\sigma \in \Pi$, define the unitary operator $\tilde{\rho}_{\sigma}(y):\mathfrak{L} \rightarrow \mathfrak{L}$ by \[ (\tilde{\rho}_{\sigma}(y)h)_j=\rho_{\sigma+j}(y)\circ h_j, \quad h \in \mathfrak{L},\] where $\rho_{\sigma+j}(y)\circ h_j$ denotes function composition. Also, define $\tilde{\rho}(y): L^2(\Pi, \mathfrak{L}) \rightarrow L^2(\Pi, \mathfrak{L})$ by 
\[ (\tilde{\rho}(y) a )(\sigma)=\tilde{\rho}_{\sigma}(y)a(\sigma), \quad a \in L^2(\Pi, \mathfrak{L}).\]
Note that if $\gamma \in \Gamma_0$, then $(\tilde{\rho}(\gamma) a )(\sigma)=e^{2 \pi i \langle \sigma, \gamma  \rangle} a(\sigma)$ for all $a \in L^2(\Pi, \mathfrak{L}).$ Further, the mapping $T$ is unitary, and for each $y \in \H$, we have \[ T(L_y \phi)(\sigma)=(\tilde{\rho}(y) T\phi )(\sigma). \]
Proofs of these results and a more detailed study of these operators can be found in (\cite{cur14}, Section 3). Fix a discrete subset $\Gamma$ of $\H$ of the form $\Gamma_1 \Gamma_0.$ Let $\mathcal{V} \subset L^2(\H)$ be a countable set. Define $E(\mathcal{V})=\{L_{\gamma}\phi: \gamma \in \Gamma, \phi \in \mathcal{V}\}$ and put $\mathcal{S}=\overline{\mathrm{span}}\; E(\mathcal{V})$. Let $R$ be the range function associated with $\mathcal{S}$. Motivated by the results of Currey et al. \cite{cur14}, we obtain the following.
\begin{lem}
Let $f \in L^2(\H)$, $\Gamma \subseteq \H$ and $\mathcal{V} \subset L^2(\H)$. Then
\[ \sum_{\phi \in \mathcal{V}, \gamma \in \Gamma}|\langle f, L_{\gamma}\phi \rangle|^2=\int_{\Pi}\sum_{\phi \in \mathcal{V}, k \in \Gamma_1} |\langle Tf(\sigma), T(L_k\phi)(\sigma) \rangle|^2 d\sigma. \]
\end{lem}
\begin{proof}
Let $ f \in L^2(\H)$. Since $\|Tf\|=\|f\|$, we have 
\begin{eqnarray*}
\sum_{\phi \in \mathcal{V}, \gamma \in \Gamma}|\langle f, L_{\gamma}\phi \rangle|^2
=\sum_{\phi \in \mathcal{V}, \gamma \in \Gamma}|\langle Tf, T(L_{\gamma}\phi) \rangle|^2
&=& \sum_{\phi \in \mathcal{V}, \gamma \in \Gamma}\left| \int_{\Pi} \langle Tf(\sigma), T(L_{\gamma}\phi)(\sigma) \rangle \; d \sigma\right|^2 \\
&=& \sum_{\phi \in \mathcal{V}, \gamma \in \Gamma}\left| \int_{\Pi} \langle Tf(\sigma), (\tilde{\rho}(\gamma) T\phi)(\sigma) \rangle \; d \sigma\right|^2.
\end{eqnarray*}
Putting $\gamma=kl$, with $k \in \Gamma_1,  l \in \Gamma_0$, we get 
\[ (\tilde{\rho}(kl) T\phi)(\sigma)=\tilde{\rho}_{\sigma}(kl) T\phi(\sigma)=\tilde{\rho}_{\sigma}(k)\tilde{\rho}_{\sigma}(l) T\phi(\sigma)=e^{2 \pi i \langle \sigma,l \rangle}\tilde{\rho}_{\sigma}(k)T\phi(\sigma). \]
Thus
\begin{eqnarray*}
\sum_{\phi \in \mathcal{V}, \gamma \in \Gamma}\left| \int_{\Pi} \langle Tf(\sigma), (\tilde{\rho}(\gamma) T\phi)(\sigma) \rangle \; d \sigma\right|^2 
= \sum_{\phi \in \mathcal{V}, (k,l) \in \Gamma}\left| \int_{\Pi} \langle Tf(\sigma), \tilde{\rho}_{\sigma}(k)T\phi(\sigma) \rangle e^{-2 \pi i \langle \sigma,l \rangle} d \sigma\right|^2.
\end{eqnarray*}
For each $k$, define $F_k(\sigma)=\langle Tf(\sigma), \tilde{\rho}_{\sigma}(k)T\phi(\sigma) \rangle .$ Then $F_k$ is integrable with square summable Fourier coefficients, hence $F_k \in L^2(\Pi)$. Using Fourier inversion formula we obtain 
\begin{eqnarray*}
\sum_{\phi \in \mathcal{V}, (k,l) \in \Gamma}\left| \int_{\Pi} \langle Tf(\sigma), \tilde{\rho}_{\sigma}(k)T\phi(\sigma) \rangle e^{-2 \pi i \langle \sigma,l \rangle} d \sigma\right|^2 
&=&\sum_{\phi \in \mathcal{V}, (k,l) \in \Gamma} |\hat{F_k}(l)|^2\\
&=& \sum_{\phi \in \mathcal{V}, k \in \Gamma_1} \|F_k\|^2 \\
&=& \sum_{\phi \in \mathcal{V}, k \in \Gamma_1} \int_{\Pi} |F_k(\sigma)|^2 d\sigma.
\end{eqnarray*}
Again, by substituting $F_k(\sigma)=\langle Tf(\sigma), \tilde{\rho}_{\sigma}(k)T\phi(\sigma) \rangle$ in the above we get
\begin{eqnarray}\label{eq1}
\sum_{\phi \in \mathcal{V}, \gamma \in \Gamma}|\langle f, L_{\gamma}\phi \rangle|^2 =\sum_{\phi \in \mathcal{V}, k \in \Gamma_1} \int_{\Pi} |F_k(\sigma)|^2 d\sigma 
& =& \sum_{\phi \in \mathcal{V}, k \in \Gamma_1} \int_{\Pi} |\langle Tf(\sigma), \tilde{\rho}_{\sigma}(k)T\phi(\sigma) \rangle|^2 d\sigma \nonumber \\
&= & \int_{\Pi} \sum_{\phi \in \mathcal{V}, k \in \Gamma_1}  |\langle Tf(\sigma), T(L_k\phi)(\sigma) \rangle|^2 d\sigma.
\end{eqnarray}
This completes the proof.
\end{proof}
Now we are in a position to prove our main result of this section.
\begin{thm}
If $\{ T_{\sigma}(L_k\phi): k \in \Gamma_1, \phi \in \mathcal{V} \}$ is a frame for its spanned vector space for almost every $\sigma \in \Pi$. Then $\{L_{\gamma}\phi: \gamma \in \Gamma, \phi \in \mathcal{V}\}$ is also a frame for its spanned vector space.
\end{thm}
\begin{proof}
Suppose that $f \in \mathcal{S}$, then $Tf(\sigma) \in R(\sigma)$ holds for a.e. $\sigma$. Since for a.e. $\sigma \in \Pi$, $\{ T_{\sigma}(L_k\phi): k \in \Gamma_1, \phi \in \mathcal{V} \}$ is a frame for its spanned vector space, there exist $0<A \leq B<\infty$ such that
\[ A\|Tf(\sigma)\|^2 \leq \int_{\Pi}\sum_{\phi \in \mathcal{V}, k \in \Gamma_1}|\langle Tf(\sigma), T(L_k\phi)(\sigma) \rangle|^2 d\sigma \leq  B\|Tf(\sigma)\|^2. \]
holds for a.e. $\sigma$. Integrating over $\Pi$ yields 
\begin{eqnarray*}
A\|f\|^2=A\|Tf\|^2=A \int_{\Pi}\|Tf(\sigma)\|^2 d\sigma &\leq & \int_{\Pi}\sum_{\phi \in \mathcal{V}, k \in \Gamma_1}|\langle Tf(\sigma), T(L_k\phi)(\sigma) \rangle|^2 d\sigma \\ 
&\leq & B \int_{\Pi}\|Tf(\sigma)\|^2 d\sigma = B\|f\|^2.
\end{eqnarray*}
Using (\ref{eq1}) we obtain
\[ A\|f\|^2 \leq \sum_{\phi \in \mathcal{V}, \gamma \in \Gamma}|\langle f, L_{\gamma}\phi \rangle|^2 \leq B\|f\|^2. \]
Hence, we have the desired result.
\end{proof}
\begin{rmk}
Notice that the family $\{ T_{\sigma}(L_k\phi): k \in \Gamma_1, \phi \in \mathcal{V} \}$ constitutes a frame for the space which consists of  all functions of the form $T_\sigma f$ for every $f \in L^2(\H)$. Similarly, the above result can be extended for semi-continuous $g$-frames using the Riesz representation theorem.
\end{rmk}

\section{Identities and inequalities for semi-continuous g-frames}\label{sec3}
Let $\{ \Lambda_{x,j} \in \mathcal{L}(\BH,\BK_{x,j}):x \in \X, j \in J \}$ be a semi-continuous $g$-frame for $\BH$ with respect to $(\X, \mu)$. Then we define the semi-continuous $g$-frame operator $S$ as follows:
\[S: \BH \to \BH, \quad Sf=\int_{\X}\sum_{j \in J} \Lambda^*_{x,j} \Lambda_{x,j} f \; d\mu(x), \]
where $\Lambda^*_{x,j}$ is the adjoint of $\Lambda_{x,j}$. It is easy to show that $S$ is a bounded, invertible, self-adjoint and positive operator. Therefore for any $f \in \BH$, we have the following reconstructions:
\[f=SS^{-1}f=\int_{\X}\sum_{j \in J} \Lambda^*_{x,j} \Lambda_{x,j} S^{-1}f \; d\mu(x),\]
\[f=S^{-1}Sf=\int_{\X}\sum_{j \in J} S^{-1} \Lambda^*_{x,j} \Lambda_{x,j} f \; d\mu(x).\]
Denote $\tilde{\Lambda}_{x,j}=\Lambda_{x,j} S^{-1}$. Then 
$\{ \tilde{\Lambda}_{x,j} :x \in \X, j \in J \}$ is also a semi-continuous $g$-frame with frame bounds $\frac{1}{B},\frac{1}{A}$, which we call the canonical dual frame of $\{ \Lambda_{x,j}:x \in \X, j \in J \}$. A semi-continuous $g$-frame $\{\gxj \in \mathcal{L}(\BH,\BK_{x,j}):x \in \X, j \in J \}$ is called an alternate dual frame of $\{ \Lambda_{x,j}:x \in \X, j \in J \}$ if for all $f \in \BH$, the following identity holds:
\begin{equation}\label{eq301}
f=\int_{\X}\sum_{j \in J} \Lambda^*_{x,j} \gxj f \; d\mu(x)=\int_{\X}\sum_{j \in J} \mathcal{G}^*_{x,j} \Lambda_{x,j} f \; d\mu(x).
\end{equation}
A semi-continuous $g$-frame $\{\Lambda_{x,j}:x \in \X, j \in J \}$ is called a Parseval semi-continuous $g$-frame, if the frame bounds $A=B=1$. For any $\X_1 \subset \X$, we denote $\X_1^c=\X \setminus \X_1$, and define the following operator:
\[ S_{\X_1}f=\int_{\X_1}\sum_{j \in J} \Lambda^*_{x,j} \Lambda_{x,j} f \; d\mu(x). \]

In \cite{bal06}, the authors proved a longstanding conjecture of the signal processing community: a signal can be reconstructed without information about the phase. While working on efficient algorithms for signal reconstruction, Balan et al. \cite{bal07} discovered a remarkable new identity for Parseval discrete frames, given in the following form.
\begin{thm}\label{th1}
Let $\{ f_j\}_{j \in J }$ be a Parseval frame for $\BH$, then for every $K\subset J$ and every $f \in \BH$, we have
\[ \sum_{j \in K} \vert \langle f,f_j \rangle \vert^2- \bigg\Vert \sum_{j \in K} \langle f,f_j \rangle f_j \bigg\Vert^2=\sum_{j \in K^c} \vert \langle f,f_j \rangle \vert^2- \bigg\Vert \sum_{j \in K^c} \langle f,f_j \rangle f_j \bigg\Vert^2. \]
\end{thm}
\begin{thm}\label{th001}
If $\{ f_j\}_{j \in J}$ be a Parseval frame for $\BH$, then for every $K\subset J$ and every $f \in \BH$, we have
\[ \sum_{j \in K} \vert \langle f,f_j \rangle \vert^2 + \bigg\Vert \sum_{j \in K^c} \langle f,f_j \rangle f_j \bigg\Vert^2 \geq \frac{3}{4} \Vert f \Vert^2. \]
\end{thm}
In fact, the identity appears in Theorem \ref{th1} was obtained in \cite{bal07} as a particular case of the following result for general frames.
\begin{thm}\label{th01}
Let $\{ f_j\}_{j \in J }$ be a frame for $\BH$ with canonical dual frame $\{ \tilde{f_j}\}_{ j \in J }$. Then for every $K\subset J$ and every $f \in \BH$, we have
\[ \sum_{j \in K} \vert \langle f,f_j \rangle \vert^2- \sum_{j \in J} \vert \langle S_K f, \tilde{f_j}  \rangle \vert^2 = \sum_{j \in K^c} \vert \langle f,f_j \rangle \vert^2- \sum_{j \in J} \vert \langle S_{K^c} f, \tilde{f_j}  \rangle \vert^2 .\]
\end{thm} 
Motivated by these interesting results, the authors in \cite{gua06, zhu10} generalized Theorems \ref{th1} and \ref{th001} to canonical and alternate dual frames. In this section, we investigate the above mentioned results for semi-continuous $g$-frames and derive some important identities and inequalities of these frames. We first state a simple result on operators.
\begin{lem}\label{lem31}
\cite{zhu10} If $P,Q \in \mathcal{L}(\BH)$ satisfying $P+Q=I$, then $P-P^{*}P=Q^*-Q^*Q.$
\end{lem}
\begin{proof}
We compute $P-P^{*}P=(I-P^{*})P=Q^*(I-Q)=Q^*-Q^*Q.$
\end{proof}
\begin{thm}
Let $\{ \Lambda_{x,j}:x \in \X, j \in J \}$ be a Parseval semi-continuous $g$-frame for $\BH$ with respect to $(\X, \mu)$. Then for every $\X_1 \subset \X$ and every $f \in \BH$, we have
\begin{eqnarray}\label{eq31}
&& \int_{\X_1}\sum_{j \in J}\Vert \Lambda_{x,j} f \Vert^2 d\mu(x) - \bigg\Vert \int_{\X_1}\sum_{j \in J} \Lambda^*_{x,j} \Lambda_{x,j} f \; d\mu(x) \bigg\Vert^2 \nonumber \\
&=& \int_{\X_1^c}\sum_{j \in J}\Vert \Lambda_{x,j} f \Vert^2 d\mu(x) - \bigg\Vert \int_{\X_1^c}\sum_{j \in J} \Lambda^*_{x,j} \Lambda_{x,j} f \; d\mu(x) \bigg\Vert^2.
\end{eqnarray}
\end{thm}
\begin{proof}
Since $\{ \Lambda_{x,j}:x \in \X, j \in J \}$ is a Parseval semi-continuous $g$-frame, the corresponding frame operator $S=I$, and hence $S_{\X_1}+S_{\X_1^c}=I$. Note that $S_{\X_1^c}$ is a self-adjoint operator, and therefore $S_{\X_1^c}^*=S_{\X_1^c}$. Applying Lemma \ref{lem31} to the operators $S_{\X_1}$ and $S_{\X_1^c}$, we obtain that for every $f \in \BH$
\begin{eqnarray}\label{eq32}
&& \langle S_{\X_1} f,f \rangle - \langle S_{\X_1}^{*}S_{\X_1} f,f \rangle =\langle S_{\X_1^c}^* f, f \rangle -\langle S_{\X_1^c}^* S_{\X_1^c}f,f \rangle \nonumber \\
\Rightarrow && \langle S_{\X_1} f,f \rangle  - \Vert S_{\X_1} f \Vert^2 = \langle S_{\X_1^c} f,f \rangle  - \Vert S_{\X_1^c}f \Vert^2.
\end{eqnarray}
We have 
\begin{eqnarray}\label{eq33}
\langle S_{\X_1} f,f \rangle =\left\langle \int_{\X_1}\sum_{j \in J} \Lambda^*_{x,j} \Lambda_{x,j} f \; d\mu(x),f \right\rangle
&=& \int_{\X_1}\sum_{j \in J}\langle \Lambda_{x,j} f, \Lambda_{x,j}f   \rangle \; d\mu(x) \nonumber \\
&=& \int_{\X_1}\sum_{j \in J} \| \Lambda_{x,j} f \|^2 d\mu(x).
\end{eqnarray}
Similarly
\begin{equation}\label{eq34}
\langle S_{\X_1^c} f,f \rangle = \int_{\X_1^c}\sum_{j \in J} \| \Lambda_{x,j} f \|^2 d\mu(x).
\end{equation}
Using equations \eqref{eq33} and \eqref{eq34} in \eqref{eq32}, we obtain the desired result.
\end{proof}
Now we generalize Theorem \ref{th1} to dual semi-continuous $g$-frames. We first need the following lemma.
\begin{lem}\label{lem32}
\cite{por17} Let $P,Q \in \mathcal{L}(\BH)$ be two self-adjoint operators such that $P+Q=I$. Then for any $\lambda \in [0,1]$ and every $f \in \BH$ we have
\[ \Vert Pf \Vert^2+2 \lambda \langle Qf,f \rangle= \Vert Qf \Vert^2+2(1-\lambda)\langle Pf,f \rangle+ (2\lambda -1)\Vert f \Vert^2 \geq (1-(\lambda-1)^2)\Vert f \Vert^2. \]
\end{lem}
\begin{thm}\label{th02}
Let $\{ \Lambda_{x,j}:x \in \X, j \in J \}$ be a semi-continuous $g$-frame for $\BH$ with respect to $(\X, \mu)$ and $\{ \tilde{\Lambda}_{x,j}:x \in \X, j \in J \}$ be the canonical dual frame of $\{ \Lambda_{x,j}:x \in \X, j \in J \}$. Then for any $\lambda \in [0,1]$, for every $\X_1 \subset \X$ and every $f \in \BH$, we have
\begin{eqnarray*}
&& \int_{\X}\sum_{j \in J} \| \tilde{\Lambda}_{x,j} S_{\X_1} f \|^2 d\mu(x)+\int_{\X_1^c}\sum_{j \in J} \| \Lambda_{x,j} f \|^2 d\mu(x) \\ 
&&=\int_{\X}\sum_{j \in J} \| \tilde{\Lambda}_{x,j} S_{\X_1^c} f \|^2 d\mu(x)+ \int_{\X_1}\sum_{j \in J} \| \Lambda_{x,j} f \|^2 d\mu(x)\\
&&\geq (2 \lambda- \lambda^2)\int_{\X_1}\sum_{j \in J} \| \Lambda_{x,j} f \|^2 d\mu(x)+(1 - \lambda^2) \int_{\X_1^c}\sum_{j \in J} \| \Lambda_{x,j} f \|^2 d\mu(x).
\end{eqnarray*} 
\end{thm}
\begin{proof}
Let $S$ be the frame operator for $\{ \Lambda_{x,j}:x \in \X, j \in J \}$. Since $S_{\X_1}+S_{\X_1^c}=S$, it follows that $S^{-1/2}S_{\X_1} S^{-1/2}+S^{-1/2}S_{\X_1^c}S^{-1/2}=I.$ 
Considering $P=S^{-1/2}S_{\X_1}S^{-1/2}$, $Q=S^{-1/2}S_{\X_1^c}S^{-1/2}$, and $S^{1/2}f$ instead of $f$ in Lemma \ref{lem32}, we obtain
\begin{eqnarray}\label{eq3001}
\nonumber
&& \Vert S^{-1/2}S_{\X_1}f \Vert^2+2 \lambda \langle S^{-1/2}S_{\X_1^c}f,S^{1/2}f \rangle \\ 
&&= \Vert S^{-1/2}S_{\X_1^c}f \Vert^2+2(1-\lambda)\langle S^{-1/2}S_{\X_1} f,S^{1/2}f \rangle+ (2\lambda -1)\Vert S^{1/2}f \Vert^2 \nonumber \\ \nonumber
&& \geq  (1-(\lambda-1)^2)\Vert S^{1/2}f \Vert^2 \nonumber \\ \nonumber
&& \Rightarrow \langle S^{-1}S_{\X_1}f, S_{\X_1}f \rangle +\langle S_{\X_1^c}f,f \rangle  = \langle S^{-1}S_{\X_1^c}f, S_{\X_1^c}f \rangle +\langle S_{\X_1} f,f \rangle \\ 
&& \geq (2 \lambda- \lambda^2) \langle S_{\X_1}f,f \rangle+(1 - \lambda^2) \langle S_{\X_1^c}f,f \rangle .
\end{eqnarray} 
We have 
\begin{eqnarray}\label{eq331}
\langle S^{-1}S_{\X_1}f, S_{\X_1}f \rangle 
&=& \langle SS^{-1}S_{\X_1}f, S^{-1}S_{\X_1}f \rangle \nonumber \\
&=& \left\langle \int_{\X}\sum_{j \in J} \Lambda^*_{x,j} \Lambda_{x,j} S^{-1}S_{\X_1}f  d\mu(x), S^{-1}S_{\X_1}f \right\rangle \nonumber \\
&=& \int_{\X}\sum_{j \in J} \langle \Lambda_{x,j} S^{-1}S_{\X_1}f, \Lambda_{x,j} S^{-1}S_{\X_1}f \rangle \; d\mu(x) \nonumber \\
&=& \int_{\X}\sum_{j \in J} \langle \tilde{\Lambda}_{x,j} S_{\X_1}f, \tilde{\Lambda}_{x,j} S_{\X_1}f \rangle \; d\mu(x) \nonumber \\
&=& \int_{\X}\sum_{j \in J} \| \tilde{\Lambda}_{x,j} S_{\X_1}f\|^2 d\mu(x).
\end{eqnarray} 
Similarly 
\begin{eqnarray} \label{eq332}
\langle S^{-1}S_{\X_1^c}f, S_{\X_1^c}f \rangle= \int_{\X}\sum_{j \in J} \| \tilde{\Lambda}_{x,j} S_{\X_1^c}f\|^2 d\mu(x).
\end{eqnarray}
\begin{eqnarray}\label{eq333}
\langle S_{\X_1^c}f,f \rangle =\int_{\X_1^c}\sum_{j \in J} \| \Lambda_{x,j} f \|^2 d\mu(x).
\end{eqnarray}
\begin{eqnarray}\label{eq334}
\langle S_{\X_1}f,f \rangle =\int_{\X_1}\sum_{j \in J} \| \Lambda_{x,j} f \|^2 d\mu(x).
\end{eqnarray}
Using equations \eqref{eq331}--\eqref{eq334} in the inequality (\ref{eq3001}), we obtain the desired result.
\end{proof}
\begin{lem}\label{lem33}
\cite{por17} If $P,Q \in \mathcal{L}(\BH)$ satisfy $P+Q=I$, then for any $\lambda \in [0,1]$ and every $f \in \BH$ we have
\[ P^{*}P+\lambda(Q^{*}+Q)=Q^{*}Q+(1-\lambda)(P^{*}+P)+(2\lambda -1)I \geq (1-(\lambda-1)^2)I. \]
\end{lem}
\begin{thm}\label{th03}
Let $\{ \Lambda_{x,j}:x \in \X, j \in J \}$ be a semi-continuous $g$-frame for $\BH$ with respect to $(\X, \mu)$ and $\{ \gxj:x \in \X, j \in J \}$ be an alternate dual frame of $\{ \Lambda_{x,j}:x \in \X, j \in J \}$. Then for any $\lambda \in [0,1]$, for every $\X_1 \subset \X$ and every $f \in \BH$, we have
\begin{eqnarray*}
&& Re \bigg\{\int_{\X_1^c}\sum_{j \in J} \langle \gxj f,  \Lambda_{x,j} f \rangle d\mu(x) \bigg\}+ \bigg\Vert \int_{\X_1}\sum_{j \in J} \Lambda^*_{x,j} \gxj f \; d\mu(x) \bigg\Vert^2 \\
&&= Re \bigg\{\int_{\X_1}\sum_{j \in J} \langle \gxj f,  \Lambda_{x,j} f \rangle d\mu(x) \bigg\}+ \bigg\Vert \int_{\X_1^c}\sum_{j \in J} \Lambda^*_{x,j} \gxj f \; d\mu(x) \bigg\Vert^2 \\
&& \geq (2 \lambda - \lambda^2) Re \bigg\{\int_{\X_1}\sum_{j \in J} \langle \gxj f,  \Lambda_{x,j} f \rangle d\mu(x) \bigg\}+(1-\lambda^2) Re \bigg\{\int_{\X_1^c}\sum_{j \in J} \langle \gxj f,  \Lambda_{x,j} f \rangle d\mu(x) \bigg\}.
\end{eqnarray*}
\end{thm}
\begin{proof}
For $\X_1 \subset \X$ and $f \in \BH$, define the operator $F_{\X_1}$ by 
\begin{equation}\label{eq35}
F_{\X_1}f=\int_{\X_1}\sum_{j \in J} \Lambda^*_{x,j} \gxj f \; d\mu(x).
\end{equation}
Then $F_{\X_1} \in \mathcal{L}(\BH)$. By (\ref{eq301}), we have $F_{\X_1}+F_{\X_1^c}=I$. By Lemma \ref{lem33}, we get
\begin{eqnarray}\label{eq36}
&& (1-(\lambda-1)^2)\Vert f \Vert^2 \leq \langle F_{\X_1}^{*}F_{\X_1} f,f \rangle +\lambda \langle(F_{\X_1^c}^{*}+F_{\X_1^c})f,f \rangle \nonumber \\ &&= \langle F_{\X_1^c}^{*}F_{\X_1^c}f,f \rangle + (1-\lambda) \langle (F_{\X_1}^{*}+F_{\X_1})f,f \rangle +(2\lambda -1)\Vert  f \Vert^2 \nonumber \\
&& \Rightarrow (2 \lambda - \lambda^2) Re(\langle If,f \rangle) \leq \Vert F_{\X_1} f \Vert^2 +\lambda ( \overline{\langle F_{\X_1^c}f,f \rangle } +\langle F_{\X_1^c}f,f \rangle ) \nonumber \\
&&= \Vert F_{\X_1^c} f \Vert^2 +(1 - \lambda) ( \overline{\langle F_{\X_1} f,f \rangle } +\langle F_{\X_1} f,f \rangle )+(2\lambda -1)\Vert  f \Vert^2 \nonumber \\
&& \Rightarrow (2 \lambda - \lambda^2) Re(\langle F_{\X_1}f,f \rangle)+(1 -\lambda^2) Re(\langle F_{\X_1^c}f,f \rangle)  \leq \Vert F_{\X_1} f \Vert^2 + Re (\langle F_{\X_1^c}f,f \rangle) \nonumber \\
&&= \Vert F_{\X_1^c} f \Vert^2 + Re (\langle F_{\X_1} f,f \rangle) .
\end{eqnarray}
We have
\begin{eqnarray}\label{eq37}
\langle F_{\X_1}f,f \rangle = \left\langle \int_{\X_1}\sum_{j \in J} \Lambda^*_{x,j} \gxj f \; d\mu(x),f \right\rangle
= \int_{\X_1}\sum_{j \in J} \langle \gxj f , \Lambda_{x,j}f \rangle d\mu(x).
\end{eqnarray}
\begin{eqnarray}\label{eq38}
\langle F_{\X_1^c}f,f \rangle = \int_{\X_1^c}\sum_{j \in J} \langle \gxj f , \Lambda_{x,j}f \rangle d\mu(x).
\end{eqnarray}
Using equations \eqref{eq37}, \eqref{eq38} and \eqref{eq35} in (\ref{eq36}), we obtain the desired inequality.
\end{proof}
Next we give a generalization of the above theorem to a more general form that does not involve the real parts of the complex numbers.
\begin{thm}\label{th5}
Let $\{ \Lambda_{x,j}:x \in \X, j \in J \}$ be a semi-continuous $g$-frame for $\BH$ with respect to $(\X, \mu)$ and $\{ \gxj:x \in \X, j \in J \}$ be an alternate dual frame of $\{ \Lambda_{x,j}:x \in \X, j \in J \}$. Then for every $\X_1 \subset \X$ and every $f \in \BH$, we have
\begin{eqnarray*}
&& \bigg(\int_{\X_1^c}\sum_{j \in J} \langle \gxj f,  \Lambda_{x,j} f \rangle d\mu(x) \bigg)+ \bigg\Vert \int_{\X_1}\sum_{j \in J} \Lambda^*_{x,j} \gxj f \; d\mu(x) \bigg\Vert^2 \\
&&= \overline{\bigg(\int_{\X_1}\sum_{j \in J} \langle \gxj f,  \Lambda_{x,j} f \rangle d\mu(x) \bigg)}+ \bigg\Vert \int_{\X_1^c}\sum_{j \in J} \Lambda^*_{x,j} \gxj f \; d\mu(x) \bigg\Vert^2.
\end{eqnarray*}
\end{thm}
\begin{proof}
For $\X_1 \subset \X$ and $f \in \BH$, we define the operator $F_{\X_1}$ as in Theorem \ref{th03}. Therefore, we have $F_{\X_1}+F_{\X_1^c}=I$. By Lemma \ref{lem31}, we have
\begin{eqnarray*}
&& \bigg(\int_{\X_1^c}\sum_{j \in J} \langle \gxj f,  \Lambda_{x,j} f \rangle d\mu(x) \bigg)+ \bigg\Vert \int_{\X_1}\sum_{j \in J} \Lambda^*_{x,j} \gxj f \; d\mu(x) \bigg\Vert^2 \\
&& = \langle F_{\X_1^c}f,f \rangle + \langle F_{\X_1}^* F_{\X_1} f,f \rangle 
= \langle F_{\X_1}^*f,f \rangle+\langle F_{\X_1^c}^* F_{\X_1^c} f,f \rangle \\
&& = \overline{ \langle F_{\X_1} f,f \rangle}+\Vert F_{\X_1^c} f \Vert^2 \\
&& = \overline{\bigg(\int_{\X_1}\sum_{j \in J} \langle \gxj f,  \Lambda_{x,j} f \rangle d\mu(x) \bigg)}+ \bigg\Vert \int_{\X_1^c}\sum_{j \in J} \Lambda^*_{x,j} \gxj f \; d\mu(x) \bigg\Vert^2.
\end{eqnarray*}
Hence the relation stated in the theorem holds.
\end{proof}

\section{Stability of semi-continuous g-frames}\label{sec4}
The stability of frames is important in practice, so it has received much attentions and is, therefore, studied widely by many authors (see \cite{chr13, pori17, sun07}). In this section, we study the stability of semi-continuous $g$-frames. The following is a fundamental result in the study of the stability of frames.
\begin{prop}
$($\cite{cas97}, $\mathrm{Theorem \;2})$
Let $\{f_i \}_{i=1}^\infty$ be a frame for some Hilbert space $\BH$ with bounds $A, B$. Let $\{g_i \}_{i=1}^\infty \subseteq \BH$ and assume that there exist constants $\lambda_1, \lambda_2, \mu \geq 0$ such that $\max(\lambda_1+\frac{\mu}{\sqrt{A}}, \lambda_2)<1$ and
\begin{equation}\label{eq41}
\left\Vert \sum_{i=1}^n c_i (f_i-g_i) \right\Vert \leq \lambda_1 \left\Vert \sum_{i=1}^n c_i f_i \right\Vert +\lambda_2 \left\Vert \sum_{i=1}^n c_i g_i \right\Vert + \mu \left[ \sum_{i=1}^n |c_i|^2 \right]^{1/2} 
\end{equation} 
for all $c_1,...,c_n(n \in \BN).$ Then $\{g_i \}_{i=1}^\infty$ is a frame for $\BH$ with bounds
\[ A \left( 1 - \frac{\lambda_1+\lambda_2+\frac{\mu}{\sqrt{A}}}{1+\lambda_2} \right)^2, \;\; B \left( 1+ \frac{\lambda_1+\lambda_2+\frac{\mu}{\sqrt{B}}}{1-\lambda_2} \right)^2.\]
\end{prop}
Similar to discrete frames, semi-continuous $g$-frames are stable under small perturbations. The stability of semi-continuous $g$-frames is discussed in the following theorem.
\begin{thm}\label{th41}
Let $\{ \Lambda_{x,j}:x \in \X, j \in J \}$ be a semi-continuous $g$-frame for $\BH$ with respect to $(\X, \mu)$, with frame bounds $A$ and $B$. Suppose that $\gm \in \mathcal{L}(\BH,\BK_{x,j})$ for any $x \in \X, \; j \in J$ and there exist constants $\lambda_1, \lambda_2, \mu \geq 0$ such that $\max(\lambda_1+\frac{\mu}{\sqrt{A}}, \lambda_2)<1$ and the following condition is satisfied
\begin{eqnarray}\label{eq43}
&& \left( \int_{\X}\sum_{j \in J} \Vert (\Lambda_{x,j}-\Gamma_{x,j})f \Vert^2 d\mu(x) \right)^{1/2} \nonumber \\
&& \leq \lambda_1 \left( \int_{\X}\sum_{j \in J} \Vert \Lambda_{x,j}(f) \Vert^2 d\mu(x) \right)^{1/2} + \lambda_2 \left( \int_{\X}\sum_{j \in J} \Vert \Gamma_{x,j}(f) \Vert^2 d\mu(x) \right)^{1/2} + \mu \| f \|,
\end{eqnarray}
for all $f \in \BH$. Then $\{\gm: x \in \X, j \in J\}$ is a semi-continuous $g$-frame for $\BH$ with respect to $(\X, \mu)$, with frame bounds
\begin{equation}\label{eq10}
A \left( 1 - \frac{\lambda_1+\lambda_2+\frac{\mu}{\sqrt{A}}}{1+\lambda_2} \right)^2, \;\; B \left( 1+ \frac{\lambda_1+\lambda_2+\frac{\mu}{\sqrt{B}}}{1-\lambda_2} \right)^2.
\end{equation}
\end{thm}
\begin{proof}
Notice that
\[\int_{\X}\sum_{j \in J} \Vert \Lambda_{x,j}(f) \Vert^2 d\mu(x) \leq B \| f \|^2. \]
From \eqref{eq43} we see that
\[ \left( \int_{\X}\sum_{j \in J} \Vert (\Lambda_{x,j}-\Gamma_{x,j})f \Vert^2 d\mu(x) \right)^{1/2} \leq \left( \lambda_1 \sqrt{B} + \mu \right)\| f \| + \lambda_2 \left( \int_{\X}\sum_{j \in J} \Vert \Gamma_{x,j}(f) \Vert^2 d\mu(x) \right)^{1/2}.  \]
Using the triangle inequality, we get
\begin{eqnarray*}
&& \left( \int_{\X}\sum_{j \in J} \Vert (\Lambda_{x,j}-\Gamma_{x,j})f \Vert^2 d\mu(x) \right)^{1/2} \\
&& \geq \left( \int_{\X}\sum_{j \in J} \Vert \Gamma_{x,j}(f) \Vert^2 d\mu(x) \right)^{1/2} - \left( \int_{\X}\sum_{j \in J} \Vert \Lambda_{x,j}(f) \Vert^2 d\mu(x) \right)^{1/2}.
\end{eqnarray*}
Hence
\begin{eqnarray*}
&& (1-\lambda_2) \left( \int_{\X}\sum_{j \in J} \Vert \Gamma_{x,j}(f) \Vert^2 d\mu(x) \right)^{1/2} \\
&& \leq \left( \lambda_1 \sqrt{B} + \mu \right)\| f \| + \left( \int_{\X}\sum_{j \in J} \Vert \Lambda_{x,j}(f) \Vert^2 d\mu(x) \right)^{1/2} 
\leq \sqrt{B} \bigg( 1 + \lambda_1 + \frac{\mu}{\sqrt{B}} \bigg)\Vert f \Vert.
\end{eqnarray*}
Therefore
\[\int_{\X}\sum_{j \in J} \Vert \Gamma_{x,j}(f) \Vert^2 d\mu(x)\leq B \left( 1+ \frac{\lambda_1+\lambda_2+\frac{\mu}{\sqrt{B}}}{1-\lambda_2} \right)^2 \Vert f \Vert^2. \]
Similarly, we can prove that
\[\int_{\X}\sum_{j \in J} \Vert \Gamma_{x,j}(f) \Vert^2 d\mu(x) \geq A \left( 1 - \frac{\lambda_1+\lambda_2+\frac{\mu}{\sqrt{A}}}{1+\lambda_2} \right)^2 \Vert f \Vert^2.  \]
This completes the proof.
\end{proof}
\begin{rmk}
In general, the inequality \eqref{eq43} does not imply that $\{\gm: x \in \X, j \in J\}$ is a semi-continuous $g$-frame regardless how small the parameters $\lambda_1, \lambda_2, \mu$ are. A counterexample for $g$-frames can be found in \cite{sun07}, and an example can be constructed similarly for semi-continuous $g$-frames. 
\end{rmk}
\begin{cor}
Let $\{\Lambda_{x,j}:x \in \X, j \in J \}$ be a semi-continuous $g$-frame for $\BH$ with respect to $(\X, \mu)$, with frame bounds $A, B$, and let $\{\gm: x \in \X, j \in J\}$ be a family in $\mathcal{L}(\BH,\BK_{x,j})$ for any $x \in \X, \; j \in J$. Assume that there exists a constant $0<M<A$ such that
\[\int_{\X}\sum_{j \in J} \Vert (\Lambda_{x,j}-\Gamma_{x,j})f \Vert^2 d\mu(x) \leq M \Vert f \Vert^2, \; \forall f \in \BH, \] then 
$\{\gm: x \in \X, j \in J\}$ is a semi-continuous $g$-frame for $\BH$ with respect to $(\X, \mu)$, with bounds $A[1-(M/A)^{1/2}]^2$ and $B[1+(M/B)^{1/2}]^2$.
\end{cor}
\begin{proof}
Let $\lambda_1=\lambda_2=0$ and $\mu=\sqrt{M}.$ Since $M< A$, $\mu/\sqrt{A}=\sqrt{M/A}<1.$ So, by Theorem \ref{th41}, $\{\gm: x \in \X, j \in J\}$ is a semi-continuous $g$-frame for $\BH$ with respect to $(\X, \mu)$, with bounds $A[1-(M/A)^{1/2}]^2$ and 
$B[1+(M/B)^{1/2}]^2$.
\end{proof}

\section*{Acknowledgments}
The author is deeply indebted to Dr. Azita Mayeli for several valuable comments and suggestions. The author is grateful to the United States-India Educational Foundation for providing the Fulbright-Nehru Doctoral Research Fellowship, and Department of Mathematics and Computer Science, the Graduate Center, City University of New York, New York, USA for its kind hospitality during the period of this work. He would also like to express his gratitude to the Norbert Wiener Center for Harmonic Analysis and Applications at the University of Maryland, College Park for its kind support.

\bibliographystyle{plain}

\end{document}